\documentclass{amsart}
\usepackage{amsfonts}
\usepackage{amsmath}
\usepackage{amssymb}
\usepackage{amsthm}

\newtheorem{theorem}{Theorem}
\newtheorem{prop}{Proposition}
\newtheorem{lemma}{Lemma}
\newtheorem{rem}{Remark}
\newtheorem{exmp}{Example}

\begin{document}

\title{Chow's theorem for linear codes}
\author{Mariusz Kwiatkowski, Mark Pankov}
\subjclass[2000]{51E22, 94B27}
\keywords{linear code, Grassmann graph, Chow's theorem}
\address{Department of Mathematics and Computer Science, 
University of Warmia and Mazury, S{\l}oneczna 54, Olsztyn, Poland}
\email{mkw@matman.uwm.edu.pl, pankov@matman.uwm.edu.pl}

\maketitle

\begin{abstract}
Let $\Gamma_{k}(V)$ be the Grassmann graph formed by $k$-dimensional subspaces of an $n$-dimensional vector space over
the finite field ${\mathbb F}_{q}$ consisting of $q$ elements and $1<k<n-1$.
Denote by $\Gamma(n,k)_q$ the restriction of the Grassmann graph to the set of all non-degenerate linear $[n,k]_q$ codes.
We describe maximal cliques of the graph $\Gamma(n,k)_q$ and show that every automorphism of this graph is induced by 
a monomial semilinear automorphism of $V$.
\end{abstract}

\section{Introduction}
Let us consider the Grassmann graph $\Gamma_{k}(V)$ formed by $k$-dimensional subspaces of an $n$-dimensional vector space $V$
over a field (not necessarily finite).
The classical Chow theorem \cite{Chow} (see also \cite{BCN,Die,P1,P2,Wan}) states that 
every automorphism of this graph is induced by a semilinear automorphism of $V$ if $1<k<n-1$ an $n\ne 2k$.
In the case when $n=2k\ge 4$, there are also the automorphisms of $\Gamma_{k}(V)$ related to semilinear isomorphisms of $V$
to the dual vector space $V^{*}$. 
For $k=1,n-1$ any two distinct vertices of the Grassmann graph are adjacent and every bijective transformation of the vertex set 
is a graph automorphism.

Results of the same nature were obtained in the discipline known as Geometry of Matrices \cite{Wan}.
Also, there are analogues of Chow's theorem for Grassmannians corresponding to Tits buildings of classical types \cite{Die,P1}.
A generalization concerning isometric embeddings of Grassmann graphs can be found in \cite{P2}.

Now, we suppose that $V$ is a vector space over the finite field $\mathbb{F}_q$ consisting of $q$ elements and
consider the restriction of the Grassmann graph to the set of all non-degenerate linear $[n,k]_q$ codes
(by \cite{TVN}, it is natural to reduce the study of linear codes to the non-degenerate case only).
Following \cite{KP} we denote this restriction by $\Gamma(n,k)_{q}$.
Note that two distinct non-degenerate linear $[n,k]_q$ codes are adjacent vertices of this graph 
if they have the maximal number of the same codewords.

In \cite{KP} the path distance on the graph $\Gamma(n,k)_{q}$ is investigated. 
In the present paper, we show that every automorphism of $\Gamma(n,k)_{q}$, $1<k<n-1$ is induced by a monomial semilinear automorphism of $V$.
Recall that two linear codes are semilinearly equivalent if there is a monomial semilinear automorphism transferring one of them to the other.
So, our result says that two non-degenerate linear $[n,k]_q$ codes are semilinearly equivalent if and only if 
they are equivalent vertices of the graph $\Gamma(n,k)_{q}$, 
i.e. there is an automorphism of this graph transferring one of these vertices to the other.

The proof of Chow's theorem is based on the description of maximal cliques in the Grassmann graph.
There are precisely two types of  maximal cliques --- so-called stars and tops. 
A star consists of all $k$-dimensional subspaces containing a fixed $(k-1)$-dimensional subspace 
and a top is formed by all $k$-dimensional subspaces contained in a fixed $(k+1)$-dimensional subspace.
Since $\Gamma(n,k)_q$ is a subgraph of the Grassmann graph, every maximal clique of $\Gamma(n,k)_q$ is the intersection of this graph 
with a certain star or a top.
However, there are maximal cliques of $\Gamma_{k}(V)$ which do not intersect $\Gamma(n,k)_q$ or intersect in non-maximal cliques.
Also, maximal cliques of the same type in $\Gamma(n,k)_q$ (the corresponding maximal cliques of $\Gamma_{k}(V)$ are of the same type)
may be of different sizes. 
The detailed description of maximal cliques of $\Gamma(n,k)_q$ will be given in Section 4.
Using this description, we prove our main result (Theorem \ref{theorem}) in Section 5.

\section{Grassmann graph}
Let $V$ be an $n$-dimensional vector space over a field $F$ (not necessarily finite). 
Denote by ${\mathcal G}_{k}(V)$ the Grassmannian consisting of all $k$-dimensional subspaces of $V$.
The {\it Grassmann graph} $\Gamma_{k}(V)$ is the graph whose vertex set is ${\mathcal G}_{k}(V)$ and 
two $k$-dimensional subspaces are adjacent vertices of this graph if their intersection is $(k-1)$-dimensional. 
In the cases when $k=1,n-1$, any two distinct vertices of the Grassmann graph are adjacent. 
It is well-known that the Grassmann graph is connected.

A bijection $l:V\to V'$, where $V'$ is a vector space over a field $F'$,
is a {\it semilinear isomorphism} if 
$$l(x+y)=l(x)+l(y)$$
for all $x,y\in V$ and there is a field isomorphism $\sigma:F \to F'$ such that 
$$l(ax)=\sigma(a)l(x)$$
for every $a\in F$ and $x\in V$.
Every semilinear automorphism of $V$ induces an automorphism of the graph $\Gamma_{k}(V)$.
In the case when $n=2k$, the Grassmann graphs $\Gamma_{k}(V)$ and $\Gamma_{k}(V^{*})$ are isomorphic by duality 
and every semilinear isomorphism of $V$ to the dual vector space $V^{*}$ induces an automorphism of $\Gamma_{k}(V)$.

\begin{theorem}[W.L. Chow \cite{Chow}]
If $1<k<n-1$, then every automorphism of the Grassmann graph $\Gamma_{k}(V)$ is induced by a semilinear automorphism of $V$
or a semilinear isomorphism of $V$ to $V^{*}$
and the second possibility is realized only in the case when $n=2k$.
\end{theorem}

For $k=1,n-1$ the above statement fails.
In these cases, every bijective transformation of ${\mathcal G}_{k}(V)$ is an automorphism of $\Gamma_{k}(V)$.

The proof of Chow's theorem is based on the description of maximal cliques in the Grassmann graph,
see, for example, \cite{BCN,Die,P1,Wan}.

Let $S$ and $U$ be incident subspaces of $V$ such that $\dim S<k<\dim U$.
We write $[S,U]_{k}$ for the set of all $k$-dimensional subspaces $X$ satisfying $S\subset X\subset U$.
In the cases when $S=0$ and $U=V$, this set will be denoted by $\langle U]_{k}$ or $[S\rangle_{k}$, respectively.
Also, this set is called a {\it line} if 
$$\dim S=k-1\;\mbox{ and }\;\dim U=k+1.$$
Note that the lines of the projective space $\Pi_{V}$ (the projective space associated to $V$) 
are subsets of type $\langle U]_{1}$, where $U$ is a $2$-dimensional subspace.
By the Fundamental Theorem of Projective Geometry \cite{Artin},
every bijective transformation of ${\mathcal G}_{1}(V)$
sending lines to subsets of lines is induced by a semilinear automorphism of $V$.

There are precisely the following two types of maximal cliques in $\Gamma_{k}(V)$:
\begin{enumerate}
\item[$\bullet$] the {\it stars} $[S\rangle_{k}$, $S\in {\mathcal G}_{k-1}(V)$,
\item[$\bullet$] the {\it tops} $\langle U]_{k}$, $U\in {\mathcal G}_{k+1}(V)$.
\end{enumerate}
The intersection of two distinct stars is empty or it contains precisely one element;
the second possibility is realized only in the case when the corresponding $(k-1)$-dimensional subspaces are adjacent vertices of $\Gamma_{k-1}(V)$.
Similarly, the intersection of two distinct tops is empty or it contains precisely one element
and the second possibility is realized only if the associated $(k+1)$-dimensional subspaces are adjacent vertices of $\Gamma_{k+1}(V)$.
The intersection of a star and a top is empty or a line; in the second case,
the corresponding $(k-1)$-dimensional and $(k+1)$-dimensional subspaces are incident.

\section{Graph of non-degenerate linear codes}
Let $\mathbb{F}_{q}$ be the finite field consisting of $q$ elements.
Consider the $n$-dimensional vector space 
$$V=\underbrace{{\mathbb F}_{q}\times\dots \times{\mathbb F}_{q}}_{n}$$
over this field.
The standard base of $V$ is formed by the vectors 
$$e_{1}=(1,0,\dots,0),\dots,e_{n}=(0,\dots,0,1).$$
Denote by $C_{i}$ the kernel of the $i$-th coordinate functional
$(x_{1},\dots,x_{n})\to x_{i}$.

A {\it linear} $[n,k]_{q}$ {\it code} is a $k$-dimensional subspaces of $V$.
Following \cite{TVN} we consider non-degenerate linear codes only.
A linear $[n,k]_{q}$ code $C\subset V$ is {\it non-degenerate} if
the restriction of every coordinate functional to $C$ is non-zero, in other words, $C$ is not contained in any coordinate hyperplane $C_{i}$.

All non-degenerate linear $[n,k]_{q}$ codes form the set 
$${\mathcal C}(n,k)_{q}=
{\mathcal G}_{k}(V)\setminus \left(\bigcup_{i=1}^{n}{\mathcal G}_{k}(C_{i}) \right).$$
Recall that the number of $k$-dimensional subspaces of $V$ is equal to the Gaussian coefficient
$$\genfrac{\lbrack}{\rbrack}{0pt}{}{n}{k}_q=
\frac{(q^n-1)(q^{n-1}-1)\cdots(q^{n-k+1}-1)}{(q-1)(q^2-1)\cdots(q^k-1)};$$
in particular, the number of $1$-dimensional subspaces 
and the number of hyperplanes in $V$ are equal to
$$\genfrac{[}{]}{0pt}{}{n}{1}_q=\genfrac{\lbrack}{\rbrack}{0pt}{}{n}{n-1}_q=
[n]_{q}=\frac{q^n-1}{q-1}.$$
Using the inclusion-exclusion principle, we can show that the set ${\mathcal C}(n,k)_{q}$ contains precisely
$$\sum_{i=0}^{n-k}(-1)^i\binom{n}{i}\genfrac{\lbrack}{\rbrack}{0pt}{}{n-i}{k}_q$$
elements.

Suppose that $1<k<n-1$ and denote by $\Gamma(n,k)_{q}$ the restriction of the Grassmann graph $\Gamma_{k}(V)$ 
to the set ${\mathcal C}(n,k)_{q}$.
The graph $\Gamma(n,k)_{q}$  is connected \cite[Proposition 1]{KP}.

A semilinear automorphism $l$ of $V$ is called {\it monomial} if $l(e_{i})=e_{\delta(i)}$,
where $\delta$ is a permutation on the set $[n]=\{1,\dots,n\}$. 
Two linear $[n,k]_{q}$ codes are {\it semilinearly equivalent} if there is a monomial semilinear  automorphism of $V$
transferring one of them to the other.
Every monomial semilinear automorphism of $V$ induces an automorphism of the graph $\Gamma(n,k)_q$. 

Our main result is the following analogue of Chow's theorem.

\begin{theorem}\label{theorem}
If $1<k<n-1$, then every automorphism of the graph $\Gamma(n,k)_{q}$ is induced by a monomial semilinear automorphism of $V$. 
\end{theorem}

\begin{rem}{\rm
If $q$ is a prime number, then the automorphism group of the field ${\mathbb F}_{q}$ is trivial. 
Otherwise, this is the cyclic group generated by the Frobenius automorphism. 
}\end{rem}

\section{Maximal cliques}
In this section, we investigate maximal cliques of the graph $\Gamma(n,k)_{q}$.
We will always suppose that $1<k<n-1$.

Let $S\in {\mathcal G}_{k-1}(V)$ and $U\in {\mathcal G}_{k+1}(V)$.
Recall that the star $[S\rangle_{k}$ consists of all $k$-dimensional subspaces containing $S$
and the top $\langle U]_{k}$ is formed by all $k$-dimensional subspaces contained in $U$.
If $S\subset U$, then the intersection of the star and the top is the line $[S,U]_{k}$.
The intersections of the star $[S\rangle_{k}$ and the top $\langle U]_{k}$ with ${\mathcal C}(n,k)_q$ are denoted by 
$[S\rangle^{c}_{k}$ and $\langle U]^{c}_{k}$, respectively.

Stars and tops of ${\mathcal G}_{k}(V)$ contain precisely 
$$[n-k+1]_q\;\mbox{ and }\;[k+1]_q$$
elements, respectively. 
Every line of ${\mathcal G}_{k}(V)$ consists of $q+1$ elements.

Since $\Gamma(n,k)_{q}$ is a subgraph of $\Gamma_{k}(V)$, 
every clique of $\Gamma(n,k)_{q}$ is a  clique of $\Gamma_{k}(V)$.
So, every maximal clique of $\Gamma(n,k)_{q}$ is the intersection of ${\mathcal C}(n,k)_{q}$ with a star or a top.
However, the intersection of a maximal clique of $\Gamma_{k}(V)$ with ${\mathcal C}(n,k)_{q}$ 
may be empty or a non-maximal clique of $\Gamma(n,k)_{q}$.

For a $(k-1)$-dimensional subspace $S$ and a $(k+1)$-dimensional subspace $U$
we will say that  $[S\rangle^{c}_{k}$ is a {\it star} of ${\mathcal C}(n,k)_{q}$ 
or $\langle U]^{c}_{k}$ is a {\it top} of ${\mathcal C}(n,k)_{q}$ only in the case when this is a maximal clique of the graph $\Gamma(n,k)_{q}$.

\subsection{Stars}
If $S$ belongs to ${\mathcal C}(n,k-1)_{q}$, then the star $[S\rangle_{k}$ is contained in ${\mathcal C}(n,k)_{q}$, 
i.e. $[S\rangle_{k}=[S\rangle^{c}_{k}$ is a star of ${\mathcal C}(n,k)_{q}$. 
In what follows, every such star will be called {\it maximal}.
If $S$ does not belong to ${\mathcal C}(n,k-1)_{q}$, then $[S\rangle^{c}_{k}$ is a proper subset of $[S\rangle_{k}$.

Denote by $c(S)$ the number of coordinate hyperplanes containing a subspace $S$.
For a $(k-1)$-dimensional subspace this number is not greater than $n-k+1$. 

\begin{lemma}\label{lemma-star1}
If $S$ is a $(k-1)$-dimensional subspace of $V$ which does not belong to ${\mathcal C}(n,k-1)_q$, then 
$$|[S\rangle^{c}_{k}|=(q-1)^{c(S)-1}\cdot q^{n-k-c(S)+1}.$$
\end{lemma}

\begin{proof}
Consider a generator matrix $M$ for the subspace $S$, i.e. a matrix whose rows form a base of $S$.
Without loss of generality we can suppose that 
$M=\left[
\begin{array}{cc}
I_{k-1}& A\\
\end{array}
\right]$,
where $I_{k-1}$ is the identity $(k-1)$-matrix and the last $c(S)$ columns of the matrix $A$ are zero.
It must be pointed out that each of the remaining columns of $A$ is non-zero.
A generator matrix for any $k$-dimensional subspace containing $S$ can be obtained from $M$ by adding a certain row vector $v$
whose first $k-1$ coordinates are $0$ and $v$ is not a linear combination of the rows from $M$.
The corresponding $k$-dimensional subspace will be denoted by $X(v)$.
This is an element of ${\mathcal C}(n,k)_{q}$ only in the case when the last $c(S)$ coordinates of $v$ are distinct from $0$.
The remaining $n-(k-1)-c(S)$ coordinates may be arbitrary. 
So, there are precisely $$(q-1)^{c(S)}\cdot q^{n-k-c(S)+1}$$ possibilities for $v$. 
We have $X(v)=X(w)$ if and only if $w$ is a scalar multiple of $v$.
Therefore, there are precisely 
$$\frac{(q-1)^{c(S)}\cdot q^{n-k-c(S)+1}}{q-1}=(q-1)^{c(S)-1}\cdot q^{n-k-c(S)+1}$$
distinct element of ${\mathcal C}(n,k)_{q}$ containing $S$.
\end{proof}

\begin{prop}\label{prop-star1}
If $q\ge 3$, then $[S\rangle^{c}_{k}$ is a star of ${\mathcal C}(n,k)_q$ for every $(k-1)$-dimen\-sional subspace $S$.
\end{prop}

\begin{proof}
The statement is obvious if $S$ is an element of ${\mathcal C}(n,k-1)_{q}$.
Consider the case when $S$ does not belong to ${\mathcal C}(n,k-1)_{q}$.
Lemma \ref{lemma-star1} shows that
\begin{equation}\label{eq1}
|[S\rangle^{c}_{k}|\ge (q-1)^{n-k}\ge (q-1)^{2}\ge q+1.
\end{equation}
Suppose that $[S\rangle^{c}_{k}$ is not a star of ${\mathcal C}(n,k)_{q}$, i.e.
it is a non-maximal clique of $\Gamma(n,k)_q$. 
A maximal clique of $\Gamma(n,k)_{q}$ containing $[S\rangle^{c}_{k}$ is a top $\langle U]^{c}_{k}$
(since the intersection of two distinct stars of ${\mathcal G}_{k}(V)$ contains not greater than one element).
Therefore, $[S\rangle^{c}_{k}$ is a subset of the line $[S,U]_{k}$.
By our assumption, $S$ is contained in at least one coordinate hyperplane $C_{i}$.
The corresponding $k$-dimensional subspace $U\cap C_{i}$ belongs to the line $[S,U]_{k}$, but it is not an element of ${\mathcal C}(n,k)_q$.
Since every line of ${\mathcal G}_{k}(V)$ consists of $q+1$ elements, we have
$$|{\mathcal C}(n,k)_{q}\cap [S,U]_{k}|\le q.$$
The line $[S,U]_{k}$ contains $[S\rangle^{c}_{k}$ and the latter inequality contradicts \eqref{eq1}.
\end{proof}

\begin{prop}\label{prop-star2}
If $q=2$ and $S$ is a $(k-1)$-dimensional subspace, then $[S\rangle^{c}_{k}$ is a star of ${\mathcal C}(n,k)_{q}$
only in the case when $c(S)\le n-k-1$.
\end{prop}

\begin{proof}
Suppose that $S$ does not belong to ${\mathcal C}(n,k-1)_{q}$ and $c(S) \le n-k-1$.
Lemma \ref{lemma-star1} implies that
$$
|[S\rangle^{c}_{k}|\ge q^2=4>3=q+1.
$$
As in the proof of Proposition \ref{prop-star1}, we establish that $[S\rangle^{c}_{k}$ is not contained in a line of ${\mathcal G}_{k}(V)$
which guarantees that $[S\rangle^{c}_{k}$ is a star of ${\mathcal C}(n,k)_{q}$.

Consider the case when $c(S)=n-k$. 
As in the proof of Lemma \ref{lemma-star1}, we suppose that a generator matrix of $S$ is 
$M=\left[
\begin{array}{cc}
I_{k-1}& A\\
\end{array}
\right]$
and the last $n-k$ columns of the matrix $A$ are zero.
We take the vectors $v$ and $w$ whose first $k-1$ coordinates are $0$ and whose last $n-k$ coordinates are $1$,
and also the $k$-coordinates of $v$ and $w$ are equal to $1$ and $0$, respectively.
Then $[S\rangle^{c}_{k}$ consists of $X(v)$ and $X(w)$.
Let $X$ be the $k$-dimensional subspace spanned by the vectors
$$v,w+v_{1},\dots,w+v_{k-1},$$
where $v_{1},\dots,v_{k-1}$ are the rows of the matrix $M$.
This is an element of ${\mathcal C}(n,k)_q$ contained in the $(k+1)$-dimensional subspace 
$$U:=X(v)+X(w).$$
On the other hand, $X$ does not contain $S$. 
Therefore, $[S\rangle^{c}_{k}$  is a proper subset of $\langle U]^{c}_{k}$.
This means that $[S\rangle^{c}_{k}$ is not a star of ${\mathcal C}(n,k)_q$.

If $c(S)=n-k+1$, then $[S\rangle^{c}_{k}$ contains only one element 
and cannot be a star of ${\mathcal C}(n,k)_q$.
\end{proof}

\subsection{Tops}
Let $U$ be a $(k+1)$-dimensional subspace of $V$.
If it does not belong to ${\mathcal C}(n,k+1)_q$, then $\langle U]^{c}_{k}$ is empty.
Suppose that $U\in {\mathcal C}(n,k+1)_q$. 
Then every $k$-dimensional subspace $U\cap C_{i}$ is an element of the top $\langle U]_{k}$ which is not contained in $\langle U]^{c}_{k}$.
The subspaces 
$$U\cap C_{1},\dots, U\cap C_{n}$$
need not to be mutually distinct, but this collection contains at least $k+1$ distinct elements.
The latter follows from the fact that $U^{*}$ is spanned by the restrictions of the coordinate functionals to $U$.
Since the top $\langle U]_{k}$ contains precisely $[k+1]_{q}$ elements, we get the following lemma.

\begin{lemma}\label{lemma-top1}
If $U\in {\mathcal C}(n,k+1)_q$, then 
$$\max\{0,[k+1]_{q}-n\}\le |\langle U]^{c}_{k}|\le [k+1]_{q}-k-1.$$
\end{lemma}

\begin{prop}\label{prop-top1}
If $2k>n$, then $\langle U]^{c}_{k}$ is a top of ${\mathcal C}(n,k)_q$ for every $U\in {\mathcal C}(n,k+1)_q$.
\end{prop}

\begin{proof}
Since $2k>n$, we have
$$[k+1]_q-n-(q+1)\geq \frac{q^{k+1}-q^{2}}{q-1}-n=\frac{q^{2}(q^{k-1}-1)}{q-1}-n\geq $$
$$\geq q^{2}q^{k-2}-n=q^k-n\geq 2k-n>0.$$
So, the number of elements in $\langle U]^{c}_{k}$ is greater than 
the number of elements in a line of ${\mathcal G}_{k}(V)$, i.e.
$\langle U]^{c}_{k}$ is not contained in a line of ${\mathcal G}_{k}(V)$ which gives the claim.
\end{proof}

\begin{lemma}\label{lemma-top2}
Let $M$ be a generator matrix of $U\in {\mathcal C}(n,k+1)_q$, i.e. a matrix whose rows form a base for $U$.
Then $U\cap C_{i}$ coincides with $U\cap C_{j}$ if and only if the $i$-column of $M$ is a scalar multiple of the $j$-column.
\end{lemma}

\begin{proof}
Let $v_{1},\dots,v_{k+1}$ be the rows of the matrix $M$.
Consider the $k$-dimen\-sional subspace $U\cap C_i$ and denote by $r$ the smallest number for which 
$v_r$ does not belong to $U\cap C_i$. 
Let $M$ and $L$ be the sets of all $p\in\{r+1,\dots, k+1\}$ such that $v_{p}$ belongs to $U\cap C_i$
or does not belong to $U\cap C_i$, respectively.
Suppose that 
$$M=\{m_{1},\dots,m_{t}\}\;\mbox{ and }\;L=\{l_{1},\dots,l_{s}\}.$$
For every $p\in \{1,\dots,s\}$ the subspace $\langle v_r, v_{l_p}\rangle$ is not contained in $U\cap C_i$, 
but there is the unique non-zero scalar $a_p$ such that $v_r+ a_pv_{l_p}$ belongs to $U\cap C_i$.
The vectors
\begin{equation}\label{eq2}
v_1, \dots, v_{r-1}, v_r+a_1v_{l_1}, \dots, v_r+a_sv_{l_s}, v_{m_1}, \dots, v_{m_t}
\end{equation}
form a base for $U\cap C_i$. 
Then the $i$-coordinates of vectors 
$$v_1, \dots, v_{r-1}, v_{m_1}, \dots, v_{m_t}$$
are $0$.  
Since $U$ is not contained in $C_{i}$, the $i$-coordinate of $v_{r}$ is a non-zero scalar $a$ 
and the $i$-coordinate of every $v_{l_p}$ is equal to $-aa^{-1}_{p}$.

Similarly, we establish that the subspace $U\cap C_j$ is spanned by 
\begin{equation}\label{eq3}
v_1, \dots, v_{r'-1}, v_{r'}+b_1v_{l'_1}, \dots, v_{r'}+b_sv_{l'_{s'}}, v_{m'_1}, \dots, v_{m'_{t'}}.
\end{equation}
The $i$-column and the $j$-column of $M$ are formed by the $i$-coordinates and $j$-coordi\-na\-tes of $v_{1},\dots,v_{k+1}$,
respectively. 
The vectors \eqref{eq2} and the vectors \eqref{eq3} span the same $k$-dimensional subspace
if and only if these columns are proportional.
\end{proof}

Using Lemma \ref{lemma-top2}, we establish the existence of $U\in {\mathcal C}(n,k+1)_q$ which does not contain elements of ${\mathcal C}(n,k)_q$.
Such elements of ${\mathcal C}(n,k+1)_q$ do not define tops of ${\mathcal C}(n,k)_q$.

\begin{exmp}{\rm
Consider the linear $[7,3]_2$ code $C$ whose generator matrix is
$$\left[
\begin{array}{ccccccc}
0&0&1&0&1&1&1\\
0&1&0&1&0&1&1\\
1&0&0&1&1&0&1\\
\end{array}
\right].$$
The columns of the matrix are mutually distinct and Lemma \ref{lemma-top2} guarantees 
that the coordinate hyperplanes intersect $C$ in $7$ mutually distinct $2$-dimensional subspaces.
On the other hand, $C$ contains precisely $[3]_{2}=7$ distinct $2$-dimensional subspaces.
Therefore, it does not contain non-degenerated linear $(7,2)_2$ codes.
Similarly, if $n=[k+1]_{q}$, then we can construct a matrix of size $(k+1)\times n$ over ${\mathbb F}_{q}$ whose columns are non-zero and mutually distinct.
Let $U$ be the associated element of ${\mathcal C}(n,k+1)_q$. 
It contains precisely $n=[k+1]_{q}$ distinct $k$-dimensional subspaces.
By Lemma \ref{lemma-top2}, every $k$-dimensional subspace of $U$ is contained in a certain $C_{i}$
and $U$ does not contain elements of ${\mathcal C}(n,k)_q$.
}\end{exmp}

Now, we show that $\langle U]^{c}_k$ can be a non-empty subset in a line of ${\mathcal G}_{k}(V)$.
In this case, it is not a top of ${\mathcal C}(n,k)_q$ again.

\begin{exmp}{\rm
As above, we suppose that $M$ is a generator matrix for $U$
and $v_{1},\dots,v_{k+1}$ are the rows of this matrix.
Consider the set $W$ formed by all non-zero vectors $w=(w_{1},\dots,w_{k+1})$ whose scalar multiples 
do not appear as columns in the matrix $M$.
For every such vector $w$ we denote by $C(w)$ the $k$-dimensional subspace of $U$ 
consisting of all vectors $\sum^{k+1}_{i=1}a_{i}v_{i}$, where scalars $a_{1},\dots, a_{k+1}$ satisfy the equality 
$$\sum^{k+1}_{i=1}a_{i}w_{i}=0.$$
Since $w$ is not a scalar multiple of a certain column from $M$,
the subspace $C(w)$ is not contained in any $C_{i}$ (see the proof of Lemma \ref{lemma-top2}).
Thus, $C(w)$ is an element of ${\mathcal C}(n,k)_q$
and $\langle U]^{c}_{k}$ is formed by all $C(w)$.
If $w'$ is a scalar multiple of $w$, then $C(w)=C(w')$.
Suppose that there exist linearly independent vectors
$$t_{i}=(t_{i1},\dots,t_{ik+1}),\;\;\;i=1,\dots,k-1$$
such that the equality
$$\sum^{k+1}_{j=1}t_{ij}w_{j}=0$$
holds for every $w=(w_{1},\dots,w_{k+1})\in W$ and every $i\in \{1,\dots,k-1\}$.
Then the $(k-1)$-dimensional subspace $S$ spanned by the vectors
$$\sum^{k+1}_{j=1}t_{ij}v_{j},\;\;\;i=1,\dots,k-1$$
is contained in every $C(w)$. This means that $\langle U]^{c}_{k}$ is contained in the line $[S,U]_k$.
If $[S\rangle^{c}_k$ is a star of ${\mathcal C}(n,k)_q$,
then $\langle U]^{c}_{k}$ is a proper subset of this star.
Consider, for example, the non-degenerate linear $[12,4]_2$  code $C$ whose generator matrix is
$$\left[
\begin{array}{cccccccccccccc}
0&0&0&1&0&1&0&1&0&1&1&1\\
0&0&1&0&1&0&1&0&1&0&1&1\\
0&1&0&0&0&0&1&1&1&1&0&1\\
1&0&0&0&1&1&0&0&1&1&1&0\\
\end{array}
\right].$$
The matrix does not contain the columns
$$\left[
\begin{array}{c}
0\\
0\\
1\\
1\\
\end{array}
\right],
\left[
\begin{array}{c}
1\\
1\\
0\\
0\\
\end{array}
\right],
\left[
\begin{array}{c}
1\\
1\\
1\\
1\\
\end{array}
\right]$$
which satisfy the linear equalities
$$w_1+w_2=0,\;\; w_3+w_4=0.$$
This means that every element of $\langle C]^{c}_{3}$ contains the $2$-dimensional subspace $L$
spanned by $v_1+v_2$ and $v_3+v_4$. 
Then $\langle C]^{c}_{3}$ is contained in the line $[L,C]_3$.
Since each of these subsets contains precisely $3$ elements, they are coincident.
}\end{exmp}

\section{Proof of Theorem \ref{theorem}}

\subsection{Two lemmas and reduction to the case $k=2$}
As in the previous section, we suppose that $1<k<n-1$.
Every automorphism of $\Gamma(n,k)_{q}$ transfers maximal cliques of $\Gamma(n,k)_{q}$ to maximal cliques.
Recall that maximal cliques of this graph are stars and tops of ${\mathcal C}(n,k)_q$.

\begin{lemma}\label{lemma-clique1}
Every automorphism of $\Gamma(n,k)_{q}$ transfers maximal stars of ${\mathcal C}(n,k)_q$ to maximal stars.
\end{lemma}

\begin{proof}
Every maximal star of ${\mathcal C}(n,k)_q$ contains precisely $[n-k+1]_q$ elements
and, by Lemma \ref{lemma-top1}, a top of ${\mathcal C}(n,k)_q$ contains $[k+1]_q-l$ elements,
where $l\in \{k+1,\dots,n\}$.
In the case when $2k\le n$, we have 
$$[n-k+1]_{q}\ge [k+1]_{q}$$
and the number of elements in a maximal star is greater 
than the number of elements in a top of ${\mathcal C}(n,k)_q$.
If $2k>n$, then
$$[k+1]_q-l-[n-k+1]_q\geq \frac{q^{k+1}-q^{n-k+1}}{q-1}-n=\frac{q^{n-k+1}(q^{2k-n}-1)}{q-1}-n\geq $$
$$\geq q^{n-k+1}q^{2k-n-1}-n=q^k-n\geq 2k-n>0,$$
i.e. the number of elements in any top of ${\mathcal C}(n,k)_q$ is greater than the number of elements in a maximal star.
Therefore, an automorphism of $\Gamma(n,k)_{q}$ cannot transfer maximal stars to tops of ${\mathcal C}(n,k)_q$.
Since the number of elements in maximal stars of ${\mathcal C}(n,k)_q$ is greater than 
the numbers of elements in non-maximal stars, maximal stars go to maximal stars.
\end{proof}

\begin{lemma}\label{lemma-clique2}
Every automorphism of $\Gamma(n,k)_{q}$ transfers stars of ${\mathcal C}(n,k)_q$ to stars.
\end{lemma}

\begin{proof}
Let $f$ be an automorphism of $\Gamma(n,k)_{q}$.
Suppose that $[S\rangle^{c}_{k}$ is a non-maximal star ($S$ does not belong to ${\mathcal C}(n,k-1)_q$)
such that $f([S\rangle^{c}_{k})$ is a certain top $\langle U]^{c}_{k}$.
There exists $S'\in {\mathcal C}(n,k-1)_{q}$ such that the star $[S\rangle^{c}_{k}$ and the maximal star $[S'\rangle_{k}$
have a non-empty intersection. This intersection consists of the $k$-dimensional subspace $S+S'$.
Hence the maximal star
$$f([S'\rangle_{k})=[S''\rangle_{k}$$
intersects the top $\langle U]^{c}_{k}$ precisely in one element.
Then $S''$ is contained in $U$ and every element of the line $[S'',U]_{k}$ belongs to ${\mathcal C}(n,k)_q$.
So, the intersection of $[S''\rangle_{k}$ and $\langle U]^{c}_{k}$ is the line $[S'',U]_{k}$
and we get a contradiction.
\end{proof}

Suppose that $f$ is an automorphism of $\Gamma(n,k)_{q}$ and $k\ge 3$.
By Lemma \ref{lemma-clique1}, $f$ sends maximal stars to maximal stars and the same holds for the inverse automorphism $f^{-1}$.
Therefore, $f$ induces a bijective transformation $f_{k-1}$ of ${\mathcal C}(n,k-1)_q$
such that 
$$
f([S\rangle_{k})=[f_{k-1}(S)\rangle_{k}
$$
for every $S\in {\mathcal C}(n,k-1)_q$.
Then 
$$f_{k-1}(\langle U]^{c}_{k-1})=\langle f(U)]^{c}_{k-1}$$
for every $U\in {\mathcal C}(n,k)_q$. 
This guarantees that $f_{k-1}$ is an automorphism of $\Gamma(n,k-1)_q$
(two distinct elements of ${\mathcal C}(n,k-1)_q$ are adjacent vertices of $\Gamma(n,k-1)_{q}$ if and only if 
there is an element of ${\mathcal C}(n,k)_q$ containing them).

We apply the above arguments to $f_{k-1}$ if $k-1\ge 3$.
Step by step, we construct a sequence $f_{k-1},\dots, f_{2}$, 
where every $f_{i}$ is an automorphism of $\Gamma(n,i)_q$.
As above, we have
$$f_{i}(\langle X]^{c}_{i})=\langle f_{i+1}(X)]^{c}_{i}$$
for every $X\in{\mathcal C}(n,i+1)_q$.
This implies that
$$f_{2}(\langle X]^{c}_{2})=\langle f(X)]^{c}_{2}$$
for every $X\in {\mathcal C}(n,k)_q$.
If $f_{2}$ is induced by a semilinear automorphism $l$ of $V$, then
$$f_{2}(\langle X]^{c}_{2})=\langle l(X)]^{c}_{2}$$
for every $X\in {\mathcal C}(n,k)_q$ which means that $f(X)=l(X)$, i.e.
$f$ is induced by $l$.
So, we need to prove Theorem \ref{theorem} for the case when $k=2$.

\subsection{Proof of Theorem \ref{theorem} for $q\ge 3$ and $k=2$}
Let $f$ be an automorphism of $\Gamma(n,2)_{q}$ and $q\ge 3$.
By Lemma \ref{lemma-clique2}, $f$ transfers every star of ${\mathcal C}(n,2)_q$ to a star and 
the same holds for the inverse automorphism $f^{-1}$.
Proposition \ref{prop-star1} shows that $f$ induces a bijective transformation $g$ of ${\mathcal G}_{1}(V)$
such that 
$$f([P\rangle^{c}_{2})=[g(P)\rangle^{c}_{2}$$
for every $P\in {\mathcal G}_{1}(V)$.
Then 
$$g(\langle S]_{1})=\langle f(S)]_{1}$$
for every $S\in {\mathcal C}(n,2)_q$.
In other words, $g$ sends every line of the projective space $\Pi_{V}$ defined by an element of ${\mathcal C}(n,2)_q$ to a line.
We need to show that $g$ transfers every line of $\Pi_{V}$ to a subset of a line.
Then the Fundamental Theorem of Projective Geometry implies that $g$ is induced by a semilinear automorphism of $V$.
As in the previous subsection, we establish that $f$ is induced by the same semilinear automorphism.

It follows from Lemma \ref{lemma-star1} that 
the size of a star $[P\rangle^{c}_{2}$ depends only on the number $c(P)$, i.e.
the number of coordinate hyperplanes containing $P$.
Therefore, $g$ preserves this number, i.e.
$$c(g(P))=c(P)$$
for every $P\in {\mathcal G}_{1}(V)$.
In particular, $g$ transfers every $\langle e_{i}\rangle$ to $\langle e_{\delta(i)}\rangle$,
where $\delta$ is a permutation on $[n]$. 

\begin{lemma}\label{lemma-case3}
If $P\in {\mathcal G}_{1}(V)$ is contained in $C_{i}$, then $g(P)$ is contained in $C_{\delta(i)}$. 
\end{lemma}

\begin{proof}
We take any $1$-dimensional subspace $Q$ containing a vector $\sum_{j\ne i}a_{j}e_{j}$ such that each of the $n-1$ scalars $a_{j}$ is non-zero.
Then $$S=\langle e_{i}\rangle +Q\in {\mathcal C}(n,2)_q$$ 
and we have 
$$f(S)=\langle e_{\sigma(i)}\rangle +g(Q).$$
The equality
$$c(g(Q))=c(Q)=n-1$$
implies that $g(Q)$ is spanned by a vector $\sum_{j\ne \sigma(i)}b_{j}e_{j}$, where each of the $n-1$ scalars $b_{j}$ is non-zero.

If $g(P)$ is not contained in $C_{\sigma(i)}$, then $g(P)+g(Q)$ belongs to ${\mathcal C}(n,2)_q$
and the stars $[g(P)\rangle^{c}_{2}$ and $[g(Q)\rangle^{c}_{2}$ have a non-empty intersection.
On the other hand, $P$ and $Q$ both are contained in $C_{i}$ and
the stars $[P\rangle^{c}_{2}$ and $[Q\rangle^{c}_{2}$ are disjoint.
\end{proof}

Let $S$ be a $2$-dimensional subspace which does not belong to ${\mathcal C}(n,2)_q$. 
Then it is contained in a certain $C_{i}$.
Show that $g$ sends the line $\langle S]_1$ to a subset of a line.

Let us take two distinct $1$-dimensional subspaces $P_{1},P_{2}\subset S$. 
We state that there exists $Q\in {\mathcal C}(n,1)_{q}$ such that the $2$-dimensional subspace $Q+P_{1}$
contains a certain $Q'\in {\mathcal C}(n,1)_{q}$ distinct from $Q$. 
Let $x=(x_{1},\dots,x_{n})$ be a non-zero vector of $P_1$.
Since $q\ge 3$, there exists a vector $y=(y_{1},\dots,y_{n})$, where every $y_{i}$ is non-zero and distinct from $-x_{i}$.
The $1$-dimensional subspaces 
$$Q=\langle y \rangle\;\mbox{ and }\;Q'=\langle x+y\rangle$$
are as required.
It is clear that 
$$Q+P_{1},\; Q+P_{2},\;Q'+P_{2}$$
form a clique of $\Gamma(n,2)_q$ which is not contained in a star.
Therefore, $\langle S+Q]^{c}_{2}$
is a top of ${\mathcal C}(n,2)_{q}$. 
This top intersects the maximal star $[Q\rangle_2$ in the line $[Q,S+Q]_{2}$.
Since $f([Q\rangle_2)$ is a maximal star and the intersection of two distinct stars contains at most one element, 
$f$ transfers $\langle S+Q]^{c}_{2}$ to a certain top $\langle U]^{c}_{2}$.

Let $P$ be a $1$-dimensional subspace of $S$.
Then $P+Q$ is an element of $\langle S+Q]^{c}_{2}$ and $f(P+Q)$ belongs to $\langle U]^{c}_{2}$.
We get the inclusion $g(P)\subset U$.
On the other hand, Lemma \ref{lemma-case3} states that $g(P)$ is contained in $C_{\delta(i)}$.
Therefore, $g(P)$ is in the $2$-dimensional subspace $S'=U\cap C_{\sigma(i)}$.
So, $g$ transfers the line $\langle S]_1$ to a subset of the line $\langle S']_1$.

By the Fundamental Theorem of Projective Geometry, $g$ is induced by a semilinear automorphism of $V$.
Since it sends every $\langle e_i \rangle$ to $\langle e_{\sigma(i)} \rangle$, this semilinear automorphism is monomial.

\subsection{Proof of Theorem \ref{theorem} for $q=2$ and $k=2$}
Let $f$ be an automorphism of $\Gamma(n,2)_2$ and $n\ge 4$.

For every subset $I=\{i_{1},\dots,i_{k}\}\subset [n]$
we define 
$$P_{I}:=\langle e_{i_{1}}+\dots+e_{i_{k}} \rangle .$$ 
Since $q=2$, each element of ${\mathcal G}_{1}(V)$ is a certain $P_{I}$.
By Proposition \ref{prop-star2}, the subset $[P_{I}\rangle^{c}_{2}$ is a star of ${\mathcal C}(n,2)_2$ if and only if $|I|\ge 3$.
We will write ${\mathcal X}$ for the subset of ${\mathcal G}_{1}(V)$ consisting of all $P_{I}$ satisfying this condition.
It follows from Lemma \ref{lemma-clique2} that $f$ induces a bijective transformation $g$ of ${\mathcal X}$ such that
$P_{I}\in {\mathcal X}$ is contained in $S\in {\mathcal C}(n,2)_2$ if and only if $g(P_{I})$ is contained in $f(S)$.

By Lemma \ref{lemma-star1}, two stars $[P_{I}\rangle^{c}_{2}$ and $[P_{J}\rangle^{c}_{2}$ have the same number of elements
if and only if $|I|=|J|$.
Hence, the equality $g(P_{I})=P_{J}$ implies that $|I|=|J|$. 
In particular, $g$ lives fixed $P_{[n]}$ (the $1$-dimensional subspace containing $e_{1}+\dots+e_{n}$).
Also, for every $i\in [n]$ we have 
$$g(P_{[n]\setminus \{i\}})=P_{[n]\setminus \{\delta(i)\}},$$
where $\delta$ is a permutation on the set $[n]$.
Therefore, $g$ transfers every $P_{I}\in {\mathcal X}$ to $P_{\sigma(I)}$ if $n=4$. We want to show that the same holds for $n>4$.

In this case, for any distinct $i,j\in [n]$ the set $[P_{[n]\setminus\{i,j\}}\rangle^{c}_{2}$ is a star of ${\mathcal C}(n,2)_2$.
This star has a non-empty intersection with the star $[P_{[n]\setminus\{t\}}\rangle^{c}_2$ only in the case when $t\ne i,j$.
This implies that
$$g(P_{[n]\setminus\{i,j\}})=P_{[n]\setminus\{\delta(i),\delta(j)\}}.$$
Similarly, we establish that $g$ sends every $P_{I}\in {\mathcal X}$ to $P_{\sigma(I)}$.

Let $l$ be the linear automorphism of $V$ transferring every $x_{i}$ to $x_{\sigma(i)}$.
The compo\-sition $l^{-1} f$ is an automorphism of $\Gamma(n,2)_2$ and the associated transformation of ${\mathcal X}$ is identity. 
In other words, $l^{-1}f$ leaves fixed every star of ${\mathcal C}(n,2)_2$.
In the case when $n\ge 5$, 
every element of ${\mathcal C}(n,2)_2$ is contained in at least two distinct stars.
This means that $l^{-1}f$ is identity, i.e. $f$ is induced by $l$.

Suppose that $n=4$. 
Then $S\in {\mathcal C}(4,2)_2$ is contained in precisely one star if and only if
$S=P_{I}+P_{J}$, where $I$ and $J$ are $2$-element subsets and $I\cup J=[4]$.
The remaining elements of ${\mathcal C}(4,2)_2$ are preserved by $l^{-1}f$.
Consider, for example, $$S=P_{\{1,2\}}+P_{\{3,4\}}.$$ 
It intersects 
$$S'=P_{\{1,3,4\}}+P_{\{2,3,4\}}$$ precisely in $P_{\{1,2\}}$.
Then $l^{-1}f(S)$ intersects $l^{-1}f(S')=S'$ in a certain $1$-dimensional subspace.
Since $l^{-1}f$ transfers every star of ${\mathcal C}(4,2)_2$ to itself,
this intersection cannot be $P_{\{1,3,4\}}$ or $P_{\{2,3,4\}}$; hence it is $P_{\{1,2\}}$.
Also, $S$ contains $P_{[4]}$ and the same holds for $l^{-1}f(S)$.
So, $l^{-1}f(S)$ contains both $P_{\{1,2\}}$ and $P_{[4]}$ which implies that it coincides with $S$.
The same arguments show that $l^{-1}f$ leaves fixed every $P_{I}+P_{J}$, where $I$ and $J$ are $2$-element subsets and $I\cup J=[4]$.

\end{document}